\documentclass[A4, reqno]{amsart}
\usepackage[dvips]{graphicx}
\usepackage{euscript}
\usepackage{amscd}
\usepackage{amssymb,latexsym}
\usepackage{mathrsfs}
\usepackage[all]{xy}
\usepackage{amssymb}
\usepackage[utf8]{inputenc}
\xyoption{all}

\usepackage{amsmath,amssymb,xspace,amsthm}
\newtheorem{theorem}{Theorem}[section]
\newtheorem{example}[theorem]{Example}%[section]
\newtheorem{remark}[theorem]{Remark}%[section]
\newtheorem{lemma}[theorem]{Lemma}%[section]
\newtheorem{proposition}[theorem]{Proposition}%[section]
\numberwithin{equation}{section}

\numberwithin{equation}{section}

\newcommand{\R}{{\mathbb{R}}}

\newcommand{\C}{{\mathbb{C}}}

\begin{document}

\title[A modification of the Hodge star operator on manifolds with boundary]{A modification of the Hodge star operator on manifolds with boundary}

\author[Ryszard Rubinsztein ]{Ryszard L. Rubinsztein\\ 
%2012-05-14
}
\address {Department of Mathematics, Uppsala University, Box 480,
Se-751 06 Uppsala, Sweden;}
\email{ryszard{\@@}math.uu.se}

\begin{abstract}
If $\, M\,$ is a smooth compact oriented Riemannian manifold of dimension $\, n=4k+2\,$, with or without boundary, and $\, F\,$ is a vector bundle on $\, M\,$ with an inner product and a flat connection, we construct a modification of the Hodge star operator on the parabolic cohomology $\, H^{2k+1}_{par}(M;F)\,$. This operator gives a canonical complex structure on  $\, H^{2k+1}_{par}(M;F)\,$ compatible with the symplectic form $\, \omega\,$ given by the wedge product of forms in the middle dimension. In case when $\, k=0\,$ that gives a canonical almost complex structure on the non-singular part of the moduli space of flat connections on a Riemann surface with or without boundary and monodromies along boundary components belonging to fixed conjugacy classes. The almost complex structure is compatible with the standard symplectic form $\, \omega\,$ on the moduli space.
\end{abstract}



\maketitle

\section{Introduction}\label{intro}

\bigskip

Let $\, M\,$ be a smooth compact oriented Riemannian manifold of dimension $\,n\,$, with or without boundary. Let $\, F\,$ be a smooth real vector bundle over $\, M\,$, of finite fiber dimension, equiped with a positive definite inner product $\, B\,$ and a flat connection. We denote by $\, H^*(M;F)\,$ the (deRham) cohomology of $\, M\,$ with coefficients in the local system given by $\, F\,$. 

Let $\, *:H^*(M;F) \rightarrow H^*(M;F)\,$ be the Hodge star operator given by the orientation and the Riemannian metric on $\, M\,$ (see Section \ref{hodge}).
  
For $\, n=2m\,$ the wedge product of forms and the inner product $\, B\,$ define a bilinear form $\, \omega : H^m(M;F) \otimes  H^m(M;F) \rightarrow \R\,$. If $\, n=4k+2\,$, the form $\, \omega\,$ is skew-symmetric. 

If $\, M\,$ has no boundary (and  $\, n=4k+2$), the form $\, \omega\,$ is non-degenerate and gives a symplectic structure on the vector space $\, H^{2k+1}(M;F)\,$. It is well-known that in that case the Hodge star operator $\, *\,$ gives a complex structure on  $\, H^{2k+1}(M;F)\,$ compatible with the symplectic form $\, \omega \,$. 

In the general case, when $\,M\,$ may have a non-empty boundary, we 
replace $\, H^*(M;F)\,$ by the {\it parabolic cohomology} $\, H_{par}^*(M;F)\,$  of $\, M\,$ with coefficients in the local system given by $\, F\,$ (see Section \ref{hodge}). Thus $\, H_{par}^*(M;F)\,$ is the kernel of the homomorphism of restriction to the boundary, 
\[ H_{par}^*(M;F)= \text{Ker}(r: H^*(M;F) \rightarrow H^*(\partial M;F)\, ) \,\, .  \] 
If $\, n=4k+2\,$, the restriction of the skew-symmetric form $\, \omega\,$ to the parabolic cohomology $\, H^{2k+1}_{par}(M;F)\,$ is again non-degenerate and equips it with a structure of a symplectic vector space.

It is the aim of this note to point out that, if the boundary of $\, M\,$ is possibly non-empty and $\, n=4k+2\,$, then there is a canonical modification of the Hodge star operator which gives an operator on parabolic cohomology, denoted here by $\, J_{par}\,$, 
\[J_{par}: H^{2k+1}_{par}(M;F) \rightarrow  H^{2k+1}_{par}(M;F) \,\, .    \]
The operator  $\, J_{par}\,$ satisfies  $\, J_{par}^2=-Id\,$ and  gives a complex structure on the vector space $\, H^{2k+1}_{par}(M;F)\,$  compatible with the symplectic form $\, \omega\,$ on it. When the boundary of $\, M\,$ is empty  then  $\, H_{par}^*(M;F)=  H^*(M;F)  \,$ and  $\, J_{par}\,$ is equal to the ordinary Hodge star operator.

If $\,n=2\,$ i.e. if $\, M\,$ is a compact oriented surface one can consider the moduli space $\, {\mathscr M}\,$ of flat connections on the trivial principal bundle $\, M\times G\,$, $\, G\,$ being a compact Lie group with a Lie algebra $\, {\mathfrak g}\,$. The flat connections have monodromies along boundary components restricted to fixed conjugacy classes in $\, G\,$. The moduli space $\, {\mathscr M}\,$  is a manifold with singularities. Away from the singular points, the tangent spaces to  $\, {\mathscr M}\,$ can be identified with 
the  parabolic cohomology $\, H^1_{par}(M; {\mathfrak g}_{\phi}   )\,$, where $\,  {\mathfrak g}_{\phi} \,$ is the trivial vector bundle over $\, M\,$ with fiber $\,  {\mathfrak g}\,$ and connection $\, \phi\,$. 
Let $\, \Sigma \subset {\mathscr M}\,$ denote the singular locus. The symplectic form $\, \omega\,$  is closed as a $2$-form on  $\, {\mathscr M}-\Sigma\,$ and turns it into a symplectic manifold \cite{GHJW}.  

Given a Riemannian metric on  $\, M$, the modified Hodge star operator  $\, J_{par}\,$ on  $\, H^1_{par}(M; {\mathfrak g}_{\phi} )\,$ constructed in Section \ref{hodge} gives a canonical almost complex structure on the non-sngular part of the moduli space $\, {\mathscr M}-\Sigma\,$ compatible with the symplectic form $\, \omega\,$. That applies both to the case when $\, M\,$ is with or without boundary.

%{\it Aknowledgments}: 

\bigskip

\section{A linear problem}\label{linprob}

\bigskip  

Let $\, V\,$ be a finite dimensional vector space over the field of complex numbers $\, \C\,$, equiped with a real valued positive definite inner product $\, (\,\,\, , \,\,)\,$ such that the operator of multiplication by the complex number $\, i=\sqrt{-1}\,$ is an isometry. We denote this operator by $\, J\,$. (In other words, $\, (\,\,\, , \,\,)\,$ is the real part of a hermitian inner product on  $\, V$.)  

Let $\, U\,$ be a real subspace of $\, V\,$  satisfying 
\begin{equation}\label{linprob1}  
J(U) \cap U^{\perp} = \{ 0\}.
\end{equation}
Here $\, U^{\perp}\,$ denotes the orthogonal complement of $\, U\,$ in $\, V\,$ with respect to the inner product  $\, (\,\,\, , \,\,)\,$.  The condition  (\ref{linprob1})    is equivalent to the requirement that the alternating $2$-form $\, \omega (u , v)=(Ju , v)\,$ is non-degenerate on $\, U\,$ and, hence, equips  $\, U\,$ with a structure of a symplectic space.

The aim of this Section is to make the observation that the complex structure of $\, V\,$ induces a specific complex structure on every real subspace $\, U\,$ satisfying (\ref{linprob1}). This complex structure will be compatible with the symplectic $2$-form  $\, \omega (u , v)=(Ju , v)\,$ on $\, U\,$. 

Let $\, U\,$ be a real subspace of $\, V\,$. We denote by $\, p_U:V\rightarrow U\,$  the orthogonal projection of $\, V\,$ on  $\, U\,$ and define $\, G:U\rightarrow U\,$ by $\,G(u)=p_U(J(u))\,$ for $\,u\in U\,$.  

\begin{lemma}\label{linprobL1} 
(i) For every real subspace $\, U\,$ of $\, V\,$ the real linear operator $\, G:U\rightarrow U\,$ is skew-symmetric   with respect to the inner product  $\, (\,\,\, , \,\,)\,$.

(ii)  If $\, U\,$ satisfies the condition  (\ref{linprob1}) then $\,G\,$ is invertible and the symmetric operator  $\, G^2=G\circ G:U\rightarrow U\,$ is negative definite.
\end{lemma} 

\begin{proof} (i) Let $\, u, v\in U\,$. Since $\, p_U\,$ is symmetric, while $\, J\,$ is skew-symmetric w.r.t.  $\, (\,\,\, , \,\,)\,$ on $\, V\,$, it follows that 
\[\begin{split}
(G(u), \, v)&= (p_UJ(u), \, v)=(J(u), \, p_U(v))=(J(u), \, v)=\\
&=-(u, \, J(v))=-(p_U(u), \, J(v))=-(u, \, p_UJ(v))=\\
&=-(u, \, G(v)).
\end{split} 
\] 
Thus $\,G:U\rightarrow U\,$ is skew-symmetric. 

(ii) If $\, U\,$ satisfies the condition (\ref{linprob1}) then $\, \text{Ker}(p_U)\,$ intersects the image of $\, J|_U\,$ trivially and $\, G\,$ is injective and, hence invertible. For $\, u\in U, \,u\ne 0\,$ we have
\[ 
(G^2(u), \, u)=-(G(u),\, G(u))< 0\,\, 
\]
and $\, G^2\,$ is negative definite.
\end{proof}

Let $\, U\,$ satisfy the condition (\ref{linprob1}) and let $\, R:U\rightarrow U\,$ be the positive square root of the positive definite symmetric operator $\, -G^2:U\rightarrow U, \,\, R=(-G^2)^{\frac 12}\,$. The operator $\, G\,$ commutes with $\, -G^2\,$ and maps its eigenspaces to themselves. It follows that $\, G\,$ commutes with $\, R\,$. We define the operator $\, J_U:U\rightarrow U\,$ by $\, J_U=R^{-1}G\,$. 

Let $\,  \omega (u,v)=(Ju, \, v)\,$ for $\, u,v\in U\,$.

\begin{proposition}\label{linprobP1}
If $\, U\,$ is a real subspace of $\, V\,$ satisfying the condition  (\ref{linprob1}) then the operator  $\, J_U:U\rightarrow U\,$ satisfies 

(i) $\, ( J_U)^2=-Id\,$, 

(ii)  $\, (J_U(u), \, J_U(v)) =  (u, \, v) \,\, \text{for} \,\, u,v\in U\,$,

(iii) $\, \omega (J_U(u), \, J_U(v)) =  \omega (u, \, v) \,\, \text{for} \,\, u,v\in U\,$, and 

(iv) $\, \omega (u, \, J_U(u))>0 \,$ for all $\, u\in U, \, u\ne 0\,$, 

\noindent
that is,  $\, J_U\,$ is a complex structure and an isometry on $\, U\,$, and it is compatible with the symplectic form $\, \omega\,$. 
\end{proposition}

\begin{proof}
(i)  $\, ( J_U)^2=R^{-1}GR^{-1}G=R^{-2}G^2=(-G^2)^{-1}G^2= -Id\,$. 

\medskip

(ii) Since $\, R\,$ is symmetric, $\, G\,$ is skew-symmetric and $\, R\,$ and  $\, G\,$  commute,  we have for $\, u,v\in U\,$ 
\[ \begin{split}
(J_U(u), \, J_U(v))&= (R^{-1}G(u), \, R^{-1}G(v))=( G(u), \, R^{-2}G(v))=\\
&= ( u, \, -GR^{-2}G(v))=( u, \, R^{-2}(-G^2)(v))=\\
&=(u, \, v).  
\end{split}\]

(iii)  Furthermore, we have $\, GJ_U=GR^{-1}G=J_UG \,$ and $\, J_U(v)=p_UJ_U(v)\,$ since $\, J_U(v)\in U\,$. Therefore
\[\begin{split}   \omega (J_U(u), \, J_U(v))&=  (JJ_U(u), \,  J_U(v)) = 
(JJ_U(u), \, p_UJ_U(v)) =\\
&= (p_UJJ_U(u), \, J_U(v))=(GJ_U(u), \, J_U(v))=\\
&=(J_UG(u), \, J_U(v))=(G(u), \, v)= (p_UJ(u), \, v)=\\
&=(J(u), \, v)=\\
&=\omega (u,v). 
\end{split}\]

(iv) Finally, if $\, u\in U, \, u\ne 0\,$ then  
\[\begin{split}
 \omega (u, \, J_U(u))&=(J(u),\,  J_U(u))= (p_UJ(u), \, J_U(u))=(G(u), \, J_U(u))=\\
&=(u, \, -GJ_U(u))= (u, \, -GR^{-1}G(u))=\\
&=(u, \, R^{-1}(-G^2)(u))=(u,\, R^{-1}R^2(u))=(u,\, R(u))>0  
\end{split}  \]
since $\, R\,$ is a positive definite symmetric operator on $\, U\,$.
\end{proof}

\begin{example}\label{linprobEx1} {\rm Let $\, V=\C ^2\,$ equiped with the standard inner product on $\,\C ^2\,$ identified with $\, \R ^4\,$. Choose a real number $\, r\in \R\,$. Let $\, u_1=(1, 0), \, u_2(r)=(i,r)\in V\,$ and $\, U_r=\text{span}_{\R} \{ u_1, \, u_2(r) \}\,$. Thus  $n=\dim_{\R} U_r  =2$.   Identifying $\, \C ^2\,$ with  $\, \R ^4\,$ via $\, \C ^2 \ni (z_1, z_2)\leftrightarrow (\text{Re}(z_1), \text{Im}(z_1), \text{Re}(z_2), \text{Im}(z_2)) \in \R ^4\,$ we get $\, U_r=\{ (a, b, br, 0)\, | \, a,b \in \R \}, \,\, J(U_r)=\{ (-b, a, 0, br) \, |\,\, a,b \in \R \}\,$ and $\, U_r^{\perp}= \{ (0, -cr, c, d) \,| \, c,d \in \R \}\,$. It follows that for every $\, r\in \R\,$, the real subspace $\, U_r\,$ satisfies  the condition  (\ref{linprob1}):  $\, J(U_r)\cap U_r^{\perp}=\{ 0\}\,$. If $\, r\ne 0\,$, then $\,U_r\,$ satisfies an additional property 
\begin{equation}\label{linprob2} 
 J(U_r)\cap U_r=\{ 0\},
\end{equation} 
that is, $\, U_r\,$ is a {\it totally real subspace} of $\, V\,$.
Taking direct sums of pairs $\, (V, \, U_r)\,$ one gets examples of subspaces $\, U\,$ satisfying the condition  (\ref{linprob1}) in every even  dimension $n$. 
The skew-symmetric operator $\, G:U_r \rightarrow  U_r\,$ is given by $\, G(u_1)=\frac 1{1+r^2}\,\, u_2(r)\,$ and $\, G(u_2(r))=-u_1\,$. Hence, $\, G^2=-\frac 1{1+r^2} \,\, Id_{U_r}\,$, $\, R= \frac 1{\sqrt{1+r^2}} \,\, Id_{U_r}\,$, and the complex structure $\, J_{U_r}:U_r\rightarrow U_r\,$ is given by $\, J_{U_r}(u_1)=\frac 1{\sqrt{1+r^2}}\,\, u_2(r)\,$ and  $\, J_{U_r}(u_2(r))=-\sqrt{1+r^2}\,\,  u_1\,$.

Real subspaces $\, U\,$ satisfying both properties   (\ref{linprob1}) and  (\ref{linprob2}) are typical of the geometric context in which the observations of the present Section will be applied. }
\end{example}

\bigskip

\section{Hodge theory on manifolds with boundary. \\ Modified Hodge star operator on parabolic cohomolgy }\label{hodge} 

\bigskip 

The main aim of this Section is to define a modified Hodge star operator on the parabolic cohomology (a definition of parabolic cohomolgy is recalled below). 

Let $\, M\,$ be a smooth compact oriented Riemannian manifold of dimension $\, n\,$, with or without boundary. Let $\, F\,$ be a smooth real vector bundle over $\, M$, of finite fiber dimension, equiped with a positive definite inner product $\, B( \,\, , \,\, )\,$ and a flat connection $\, A\,$.  Let $\, d_A:\Omega ^0(F)\rightarrow \Omega ^1(F) \,$ be the operator of the covariant derivative given by  $\,A\,$. Here we use $\,\Omega ^p(F)\,$ to denote smooth sections of $\, \Lambda ^pT^*M \otimes F\,$, the $p$-forms with values in $\, F\,$. We use the same symbol $\, d_A\,$ to denote the unique extension $\, d_A:\Omega ^p(F)\rightarrow \Omega ^{p+1}(F)\,$ satisfying the Leibnitz rule. Since $\, A\,$ is a flat connection, we have $\, d_Ad_A=0\,$ and get a cochain complex 
\begin{equation}\label{hodge4}
0\longrightarrow \Omega ^0(F)\xrightarrow{\, \, d_A\,\, }  \Omega ^1(F)\xrightarrow{\,\, d_A\,\, } ... \xrightarrow{\,\, d_A\,\, } \Omega ^p(F)\xrightarrow{\,\, d_A\,\, } \Omega ^{p+1}(F)\longrightarrow ...  
\end{equation} 

The Riemannian metric and the orientation on $\, M\,$ and the inner product  $\, B\,$ on $\, F\,$ give rise to an $L^2$ inner product $\, ( \,\, , \,\, )\,$ on $\,  \Omega ^*(F)\,$ satisfying 
\[ ( \alpha , \beta) = \int_M \, B(\alpha \wedge *\beta ) \,\, , \] 
where $\,*\,$ denotes the Hodge star operator. (The Hodge star operator $\, *\,$ on  $\, \Lambda ^*T^*M \otimes F\,$ is defined as the tensor product of the usual Hodge star operator on  $\, \Lambda ^*T^*M \,$ with the identity on $\, F\,$.) We have also the co-differential 
\[
\delta _A = (-1)^{n(p+1)+1} *(d_A)*: \Omega ^p(F)\rightarrow  \Omega ^{p-1}(F)\,\, ,
\]
which on closed manifolds is the $L^2$-adjoint of the operator $d_A$.

From now on the operators $\, d_A\,$ and $\,\delta _A\,$ will be denoted by $\, d\,$ and $\, \delta\,$ respectively.  

For the Hodge Decomposition Theorem on manifolds with boundary we refer to  \cite{M1} and \cite{CDGM}. The reference  \cite{CDGM} contains also short remarks on the history of the subject.  
Below we mostly quote from  \cite{CDGM}.

 A form $\, \omega\in \Omega ^p(F) \,$ is called {\it closed } if it satisfies $\, d\omega = 0\,$ and {\it co-closed} if it satisfies $\, \delta \omega = 0\,$. We denote by $\, C^p\,$  and  $\, cC^p\,$  the spaces of closed respectively co-closed $p$-forms.  We define $\, E^p=d( \Omega ^{p-1}(F)) \,$ and  $\,c E^p=\delta ( \Omega ^{p+1}(F)) \,$.

Along the boundary $\, \partial M\,$ every $p$-form $\, \omega \in  \Omega ^p(F)\,$ can be decomposed into tangential and normal components (depending on the Riemannian metric on $\, M$). For $\, x\in  \partial M\,$, one has 
\begin{equation}\label{hodge1}
\omega (x) = \omega_{tan} (x)  + \omega_{norm} (x) \,\, ,  
\end{equation} 
where  $\,  \omega_{norm} (x)\,$ belongs to the kernel of the restriction homomorphism 
\[r^*:\Lambda ^*T ^*_xM \otimes F_x \rightarrow \Lambda ^*T ^*_x( \partial M) \otimes F_x \,\, ,  \]
while  $\,  \omega_{tan} (x)\,$ belongs to the orthogonal complement of that kernel, 
\[ \omega_{tan} (x) \in (Ker(r ^*))^{\perp} \subset \Lambda ^*T ^*_xM \otimes F_x \,\, .  \]
Note that $\, r ^*\,$ maps the orthogonal complement $\, (Ker(r ^*))^{\perp}\,$ of the kernel isomorphically onto $\, \Lambda ^*T ^*_x( \partial M) \otimes F_x\,$.

Following \cite{CDGM}, we define $\, \Omega ^p_N\,$ to be the space of smooth $p$-forms from  $\,  \Omega ^p(F)\,$ satisfying {\it Neumann boundary conditions} at every point of  $\,  \partial M\,$,
\[\Omega ^p_N = \{ \omega \in  \Omega ^p(F)\, | \,  \omega_{norm}=0  \}\, ,  \]
and $\, \Omega ^p_D\,$ to be  the space of smooth $p$-forms from  $\,  \Omega ^p(F)\,$ satisfying {\it Dirichlet boundary conditions} at every point of  $\,  \partial M\,$,
\[\Omega ^p_D = \{ \omega \in  \Omega ^p(F)\, | \,  \omega_{tan}=0  \}\, .  \]
Furthermore, we define $\, cE^p_N=\delta (\Omega ^{p+1}_N)\,$ and  $\, E^p_D= d(\Omega ^{p-1}_D)\,$  and denote 
\[ CcC^p = C^p \,\cap \, cC^p= \{ \omega \in  \Omega ^p(F) \,|\,\, d\omega = 0, \,\, \delta \omega = 0\,\,\},\]
\[ CcC^p_N = \{ \omega \in  \Omega ^p(F) \,|\,\, d\omega = 0, \,\, \delta \omega = 0, \,\, \omega_{norm}=0\,\,\},\]
\[ CcC^p_D = \{ \omega \in  \Omega ^p(F) \,|\,\, d\omega = 0, \,\, \delta \omega = 0, \,\, \omega_{tan}=0\,\,\}.\]

If the boundary $\, \partial M=\emptyset \,$ is empty then every form $\, \omega\,$ satisfies $\, \omega_{norm}= \omega_{tan}=0 \,$, the space $\, CcC^p = CcC^p_N= CcC^p_D \,$ consists of all forms which are both closed and co-closed and this space is equal to the space of harmonic $p$-forms, that is, to the kernel of the Laplacian $\, \Delta = \delta d + d \delta \,$ acting on $\, \Omega ^p(F)\,$. 

If, on the other hand, the boundary  $\, \partial M \ne \emptyset\,$ is non-empty and $\, M\,$ is connected then the intersection $\, CcC^p_N\, \cap\,  CcC^p_D = 0\,$ (\cite{CDGM} Lemma 2) and the kernel of the Laplacian $\, \Delta\,$ contains all forms which are both closed and co-closed but can be strictly larger than the space of such forms, (\cite{CDGM} Example).  
 
In the following the symbol $\, \oplus \,$ will denote an orthogonal direct sum.

\begin{theorem}\label{hodgeT1} {\rm ({\bf Hodge Decomposition Theorem})} Let $\, M\,$ be a compact, connected, oriented, smooth Riemannian $n$-manifold, with or without boundary and let $\, F\,$ be a smooth real vector bundle over $\, M\,$,  of finite fiber dimension, equiped with an inner product and a flat connection $\, A\,$. Then the space  $\, \Omega ^p(F)\,$  of $\, F$-valued smooth $p$-forms decomposes into the orthogonal direct sum 
\begin{equation}\label{hodge2}
 \Omega ^p(F) = cE^p_N \oplus CcC^p \oplus E^p_D \,\, .
\end{equation} 
Furthermore, we have the orthogonal direct sum decompositions
\begin{equation}\label{hodge3}
 CcC^p =  CcC^p_N \oplus ( E^p \, \cap \, cC^p) = (C^p\, \cap \, cE^p) \oplus CcC^p_D \,\, .
\end{equation}
\end{theorem}

For the proof of Theorem \ref{hodgeT1} see \cite{M1}.  

\medskip 

We denote by $\, H^*(M; F)\,$ the cohomology of the complex (\ref{hodge4}) and define \break  $\, H^*(\partial M; F|_{\partial M})\,$ and  $\, H^*(M, \partial M; F)\,$ accordingly. 

It follows from (\ref{hodge2}) that the space $\, C^p\,$ of closed $p$-forms decomposes as $\, C^p=CcC^p \oplus E^p_D\,$. Hence, from  (\ref{hodge3}), we get  $\, C^p=CcC^p \oplus E^p_D = CcC^p_N \oplus ( E^p \, \cap \, cC^p) \oplus  E^p_D \,$. Using  (\ref{hodge3}) once again we see that $\, ( E^p \, \cap \, cC^p) \oplus  E^p_D = E^p \,$. Therefore
\begin{equation}\label{hodge5} 
 C^p=CcC^p \oplus E^p_D = CcC^p_N \oplus ( E^p \, \cap \, cC^p) \oplus  E^p_D = CcC^p_N \oplus E^p \, .  
\end{equation} 
Thus, $\, CcC^p_N\,$ is the orthogonal complement of the exact $p$-forms within the closed ones, so $\, CcC^p_N \cong  H^p(M; F)\,$. 
In a similar way, the space $\, cC^p\,$ of co-closed $p$-forms decomposes as 
\begin{equation}\label{hodge6} 
cC^p=cE^p_N \oplus CcC^p=cE^p_N\oplus (C^p\cap cE^p) \oplus CcC^p_D = cE^p \oplus CcC^p_D \, .
\end{equation}   
It follows again from (\ref{hodge2}) and  (\ref{hodge3}) that  $\, CcC^p_D \cong H^p(M,\partial M; F)\,$.  

So far the reference \cite{CDGM}.

Let now  $\, r^*:  H^*(M; F) \rightarrow  H^*(\partial M; F|_{\partial M})\,$  be the homomorphism of the restriction to the boundary.

We define {\it the parabolic cohomology  $\,H^*_{par}(M; F)\,$   of the manifold $\, M\,$ with coefficients in the bundle $\, F\,$ with the flat connection $\, A\,$} to be the kernel of the restriction homomorphism $\,  r^*\,$,  
\[ H^*_{par}(M; F):= \text{Ker} \left(\, r^*:  H^*(M; F) \rightarrow  H^*(\partial M; F|_{\partial M})\,  \right)\,.\] 
(Compare  \cite{W1} and \cite{GHJW}, Section 3.) 

Of course, the parabolic cohomology  $\,H^*_{par}(M, F)\,$ is equal to the image of \break

We assume now that the manifold $\, M\,$ has dimension  $\, n=4k+2\,$. When $\, p=2k+1\,$, the Hodge star operator $\, *\,$ maps $\, \Omega ^p(F)\,$ onto itself, $\, *: \Omega ^p(F)\rightarrow \Omega ^p(F)\,$, and satisfies $\, **=-Id\,$. 
Moreover, it maps  $\, CcC^p\,$ onto itself, mapping $\, CcC^p_N\,$ onto $\, CcC^p_D\,$ and vice-versa. Thus $\, *\,$ gives a complex structure on $\,\Omega ^p(F)\,$ and on  $\, CcC^p\,$. For the rest of this Section we shall denote the Hodge star operator $\, *\,$ on $\,\Omega ^p(F)\,$ by $\, J\,$. We have 
\begin{equation}\label{hodge7} 
J( CcC^p)= CcC^p\,,\quad  J(CcC^p_N)=CcC^p_D\quad \text{and} \quad  J(CcC^p_D)=CcC^p_N\,. 
\end{equation}   

Since $\, M\,$ is compact, the cohomology groups $\, H^p(M, F)\,$ and $\, H^p(M,\partial M ; F)\,$ and, hence, $\, CcC^p_N\,$ and $\, CcC^p_D\,$ are finite dimensional vector spaces. Let $\, P_N:CcC^p\rightarrow CcC^p_N\,$ and  $\, P_D:CcC^p\rightarrow CcC^p_D\,$ be the orthogonal projections of $\, CcC^p\,$ onto $\, CcC^p_N\,$ and  $\, CcC^p_D\,$ respectively. By (\ref{hodge3}) the kernel $\, \text{Ker}(P_N)\,$ is equal to $\, E^p\cap cC^p\,$, while the kernel  $\, \text{Ker}(P_D)\,$ is equal to $\, C^p\cap cE^p\,$. Since $\, J\,$ is an isometry of $\, CcC^p\,$, it follows from (\ref{hodge7}) that $\, P_N \circ J = J\circ P_D\,$.  Let $\, {\mathscr P}_N:CcC_D^p\rightarrow  CcC^p_N\,$ be the restriction of $\,P_N\,$ to  $\, CcC^p_D\,$ and let  $\, {\mathscr P}_D:CcC_N^p\rightarrow  CcC^p_D\,$ be the restriction of $\,P_D\,$ to  $\, CcC^p_N\,$. We have 
\begin{equation}\label{hodge8}
 {\mathscr P}_N \circ J = J\circ {\mathscr P}_D  \,.
\end{equation}

When  $\, H^p(M,\partial M ; F)\,$ is identified with $\, CcC^p_D\,$ and  
 $\, H^p(M, F)\,$ with  $\, CcC^p_N\,$, the homomorphism $\,i^*: H^*(M,\partial M ; F)\rightarrow  H^*(M, F)$ is identified with  $\, {\mathscr P}_N:CcC_D^p\rightarrow  CcC^p_N\,$. The parabolic cohomology group $\, H^p_{par}(M, F)\,$ is thus identified with the image of  $\, {\mathscr P}_N:CcC_D^p\rightarrow  CcC^p_N\,$ which we denote by $\, U, \,\, U=\text{Im}({\mathscr P}_N) \subset CcC^p_N\,$. 

It follows then from (\ref{hodge8}) that $\, J(U)\,$ is equal to the image of  $\, {\mathscr P}_D:CcC_N^p\rightarrow  CcC^p_D\,$. We denote this image by $\, T\,, \,\, T=\text{Im}({\mathscr P}_D)= J(U) \subset CcC^p_D\,$. 

Let $\, T^{\perp}\,$ be the orthogonal complement of $\, T\,$ in $\, CcC^p_D\,$. 

\begin{lemma}\label{hodgeL1}
The kernel of  $\, {\mathscr P}_N:CcC_D^p\rightarrow  CcC^p_N\,$ is equal to  $\, T^{\perp}\,$.
\end{lemma}

\begin{proof}
Let $\, w\in T^{\perp}\subset CcC^p_D\,$. Let $\, x\in CcC^p_N\,$. Since $\,P_D\,$ is a symmetric mapping and since $\,{\mathscr P}_D(x)\in T\,$,  we get   $\, (w, x)=(P_D(w), x)=(w, P_D(x))=(w,{\mathscr P}_D(x))= 0\,$. Hence $\, w\,$ is orthogonal to  $\, CcC^p_N\,$ and therefore $\, {\mathscr P}_N(w)=0\,$. Thus $\, T^{\perp}\subset \text{Ker}( {\mathscr P}_N)\,$.  On the other hand $\, \dim  T^{\perp} = \dim CcC^p_D - \dim T = \dim CcC^p_D - \dim U =\dim CcC^p_D -\dim \text{Im}( {\mathscr P}_N)  = \dim \text{Ker}( {\mathscr P}_N)\,$. Thus  $\, T^{\perp} = \text{Ker}( {\mathscr P}_N)\,$.
\end{proof}

\begin{lemma}\label{hodgeL2}
Let $\, v\in T=J(U) \,$. If $\, v\,$ is orthogonal to $\, U\,$ then $\, v=0\,$.
\end{lemma}

\begin{proof} Assume that   $\, v\in T=J(U) \,$ is orthogonal to $\, U\,$.
Since $\, v\in CcC^p_D\,$, we have $\,  {\mathscr P}_N(v)\in U=\text{Im}({\mathscr P}_N)\,$. On the other hand, since  $\,  {\mathscr P}_N\,$ is a projection along a space orthogonal to $\, CcC^p_N\,$ and, hence, orthogonal to $\, U\,$, we get that $\,  {\mathscr P}_N(v)\,$ is also orthogonal to $\, U\,$. Since  $\,  {\mathscr P}_N(v)\,$  both belongs to  $\, U\,$ and is othogonal to  $\, U\,$, we must have  $\,  {\mathscr P}_N(v)=0\,$. Thus $\, v\,$ belongs to $\,  \text{Ker}( {\mathscr P}_N)\,$ which, by Lemma \ref{hodgeL1}, is equal to $\, T^{\perp}\,$.  Belonging to $\, T\,$ and $\, T^{\perp}\,$ at the same time, $\,v\,$ must be $\, 0\,$.  
\end{proof} 

Let $\, V\,$ be the subspace of $\, CcC^p\,$ spanned by $\, CcC^p_D\,$ and  $\, CcC^p_N\,$. Since both these spaces are finite dimensional, so is $\, V$.
Moreover, (\ref{hodge7}) implies that $\, V\,$ is a complex subspace of  $\, CcC^p\,$ with respect to the complex structure $\, J$ given by the Hodge star operator. V inherits the real inner product $\, (\,\, , \,\,)\,$ from $\, CcC^p\,$ and $\, J\,$ acts as an isometry. Finally, $\, U\subset V\,$ and, according to Lemma \ref{hodgeL2},
\begin{equation}\label{hodge9}
J(U)\cap U^{\perp} = 0 \, ,
\end{equation}
where this time $\, U^{\perp}\,$ denotes the orthogonal complement of $\, U\,$ in $\, V$. 

The alternating $2$-form $\, \omega (u,v) = (J(u), v)\,$ is a symplectic (non-degenerate) form on $\, V\,$. The property  (\ref{hodge9}) implies that the restriction of $\, \omega\,$ to $\, U\,$ is a symplectic (non-degenerate) form on $\, U\,$.  

Since (\ref{hodge9}) is satisfied, we can now apply the construction of  Section \ref{linprob} to $\,V, \,\,  U\,$ and $\,J\,$ and obtain a linear operator  
\[ J_U:U \rightarrow U\, \]
which equips the space $\, U\,$ with a complex structure. 
When $\, U\,$ is identified with the parabolic cohomology $\, H^p_{par}(M; F)\,$ we denote the operator corresponding to $\, J_U\,$  by $\, J_{par}\,$ ,
\begin{equation}\label{hodge10} 
 J_{par}:  H^p_{par}(M; F) \rightarrow  H^p_{par}(M; F)
\end{equation}  
and call it {\it the modified Hodge star operator on the parabolic cohomology}.  
We have the real inner product  $\, (\,\, , \,\,)\,$ and the symplectic form $\, \omega \,$ on  $\, H^p_{par}(M; F)=U\,$. 
Proposition \ref{linprobP1} gives now

\begin{theorem}\label{hodgeT2} Let $\,M\,$ be a smooth compact oriented Riemannian manifold of dimension $\, n=4k+2\,$, with or without boundary, and $\,F\,$ be a real finite dimensional vector bundle over $\,M\,$ equiped with an inner product and a flat connection. Let $\, p=2k+1\,$. Then the modified Hodge star operator  
$\, J_{par}:  H^p_{par}(M; F) \rightarrow  H^p_{par}(M; F) \,$
satisfies 

{\rm (i)} $\, ( J_{par})^2=-Id,\,$ 

{\rm (ii)} $\, \omega (J_{par}(u), \, J_{par}(v)) =  \omega (u, \, v) \,\, \text{for} \,\, u,v\in H^p_{par}(M; F) \,$, and 

{\rm (iii)} $\, \omega(u, \, J_{par}(u))>0\,$ for all $\, u\in H^p_{par}(M; F), \, \, u\ne 0\,$,

\noindent
that is,  $\, J_{par}\,$ is a complex structure  on the parabolic cohomology  $\, H^p_{par}(M; F)\,$  compatible with the symplectic form $\, \omega\,$.
\end{theorem}

\begin{remark} {\rm 

(i) The symplectic form $\, \omega\,$ on $\, H^p_{par}(M; F)=U\,$ is the restriction of the form  $\, \omega\,$ on  $\, H^p(M; F)=CcC^p_N\,$ which in turn is given by 
\[\begin{split}\omega (u,v)&=(Ju,v)=(*u,v)= (v,*u)=\int_M B(v\wedge **u)=\int_M B(u\wedge v) =\\&=([u]\cup [v])[M; \partial M] \,,\end{split}  \] 
where $[u]$ and $[v]$ denote the cohomology classes of the closed forms $u$ and $v$. 
Thus the symplectic form  $\, \omega\,$ is given by the cup (wedge) product composed with $B$. 

(ii) When $\,M\,$ is without boundary, $\, \partial M=\emptyset\,$, then $\, CcC^p_N= CcC^p_D=U=J(U)\,$ above and $\, J_{par}=J=*\,$. Thus, in that case, the parabolic cohomology  $\, H^p_{par}(M; F)\,$ is equal to the ordinary cohomology $\,H^p(M; F)\,$ and the modified Hodge operator is equal to the ordinary Hodge star operator. 

(iii) If $\, M\,$ is not connected then it is obvious from the construction above that the parabolic cohomology  $\, H^p_{par}(M; F)\,$ and the modified Hodge operator $\,J_{par}\,$ are direct sums of their counter-parts on the components. 

(iv) The modified Hodge star operator  $\,J_{par}\,$ is canonically determined by the choice of the Riemannian metric and the orientation on $\, M\,$ and the choice of the inner product and the flat connection on $\, F\,$.  }
\end{remark}

\section{The moduli space of flat connections \\ on a Riemann surface with boundary}\label{moduli} 

\bigskip 

Let $\, G\,$ be a compact Lie group with a Lie algebra $\, {\mathfrak g}\,$ equiped with a real-valued positive definite invariant inner product. Let $\, S\,$ be a smooth compact oriented surface, with or without boundary. We consider the moduli space  $\, {\mathscr M}= {\mathscr M}(S; G, C_1, ... , C_k) \,$  of gauge equivalence classes of flat connections in the trivial pricipal $\, G$-bundle over $\, S\,$ with monodromies along boundary components belonging to some fixed conjugacy classes $\,  C_1, ... , C_k\,$    in $\, G\,$, $\,k\,$ being the number of boundary components of $\, S\,$ (see \cite{GHJW}).

The space  $\, {\mathscr M}\,$  is a finite dimensional manifold with singularities. We denote by $\, \Sigma \subset {\mathscr M}  \,$ the singular locus. Every point of  $\, {\mathscr M}\,$ can be represented by a group homomorphism $\, \phi : \pi _1 (S) \rightarrow G\,$ such that $\, \phi\,$ maps elements of $\,  \pi _1 (S)  \,$ given by the boundary components into the corresponding conjugacy classes $\, C_j\,$. Let $\, G\,$ act on $\, {\mathfrak g}\,$ through the adjoint representation. To every such group homomorphism $\,  \phi\,$ we can associate a bundle over $\, S\,$ with fiber $\, {\mathfrak g}\,$ equiped with a flat connection and an $\, \R$-valued positive definite inner product in the fibers. We denote that flat vector bundle by $\,  {\mathfrak g}_{\phi}\,$. The tangent space to  $\, {\mathscr M}\,$ at a non-singular point $\, [\phi ]\in  {\mathscr M}  \,$ is naturally identified with the parabolic cohomology group $\, H^1_{par}(S; {\mathfrak g}_{\phi})\,$ (see \cite{GHJW}, Section 3, Propositions 4.4 and 4.5 and pp.409-410 thereof). 

In \cite{GHJW} the manifold  $\, {\mathscr M}-\Sigma\,$ is equiped with a symplectic structure given by  $\, -1\,$ times the  wedge product of forms and the inner product on the bundle  $\,  {\mathfrak g}_{\phi}\,$, (\cite {GHJW}, Section 3, pp.386-387 and Theorem 10.5).
Hence, this symplectic structure is the negative of the one given by the form $\, \omega\,$ in our paper.

It follows now from Theorem \ref{hodgeT2} that a choice of a Riemannian metric on the surface $\, S\,$ gives, via the modified Hodge star operator $\, J_{par}\,$,  a canonical 
almost complex structure on the moduli space $\, {\mathscr M}-\Sigma\,$   compatible with the symplectic form $\, \omega \,$.   To get an almost complex structure on  $\, {\mathscr M}-\Sigma\,$ compatible with the symplectic form of \cite{GHJW} one has to take the operator $\,-J_{par}\,$.

\vskip1truecm 
\end{document}